\documentclass[A4]{article}
\usepackage[utf8]{inputenc}
\usepackage{graphicx}
\usepackage{amsmath}
\usepackage{amssymb, amsfonts, amscd, amsthm}
\usepackage{enumerate}
\usepackage{hyperref}
\usepackage[T2A]{fontenc}
\usepackage[english]{babel}
\newtheorem{theorem}{Theorem}
\newtheorem{proposition}{Proposition}
\newtheorem{lemma}{Lemma}
\newtheorem{corollary}{Corollary}
\usepackage{graphicx}
\usepackage{url}

\theoremstyle{remark}
\newtheorem{remark}{Remark}
\newtheorem{example}{Example}

\begin{document}

\title{Generalisation of Michelson contrast for operators and its properties}

\author{Sami Abdullah Abed$^\ast$, Irina Nikolaeva$'^{,\triangle}$, Andrei Novikov$^{\dagger,\circ}$}

\maketitle

\noindent$\ast$ E-mail: {samialbarkish@gmail.com}
{College of Administration and Economics, Diyala University, Baquba, Diyala Governorate, Republic of Iraq}

\noindent$'$ E-mail: {irina.nikolaeva@raft.ru}
{Raft Digital Solutions, Sverdlova St., 25D, Yaroslavl, Russia}

\noindent$\triangle$ E-mail: {ianikolaeva@kpfu.ru}
{Kazan Federal University, Kremlivskaya St., 35, Kazan, Russia}

\noindent$\dagger$ E-mail: {a.hobukob@gmail.com}
{Sobolev Institute of Mathematics, Siberian Branch of Russian Academy of Sciences, Novosibirsk, Russia}

\noindent$\circ$ E-mail: {a.novikov@innopolis.ru}
{Innopolis University, Universitetskaya st., 1, Innopolis, Rep. Tatarstan, 420500 Russia}

\begin{abstract} In this study we consider the generalization of the Michelson contrast for positive operators of countably decomposable $W^*$-algebras and prove its properties. In addition, we study how the inequalities characterizing traces interplay with the Michelson contrasts of operator variables.
\end{abstract}

subclass: {47B65, 47L07,	47C15} % Enter 2010 Mathematics Subject Classification.

keywords: {positive operator, invertible operator, tracial functional, C*-algebra, W*-algebra, von Neumann algebra, Michelson contrast, Jensen–Shannon divergence} % Include keywords separeted by comma.

\section{Introduction}
The article was first composed under the inner motivation of functional analysis, defining and describing interesting properties of the "centrality measure", since from various points of view, central and tracial elements are in some sense ``better'' than arbitrary elements of the algebra and functionals, respectively. A vast majority of articles dedicated to the characterization of tracial functionals by inequalities can be found in this overview. \cite{Bik2010add1}

The main inspiration was found in the 2019 article by Bikchentaev and Abed \cite{AB2019}, which contains the following results:

For $\varphi\in \mathcal{M}_*^+$ the following conditions are equivalent:

\noindent {\rm (i)} $\varphi$ trace;

\noindent {\rm (ii)} $\forall A,B\in\mathcal{M}^+$ $$\varphi(A^2+B^2+K(AB+BA))\geq0;$$

\noindent {\rm (iii)} $\exists K_0>-1\ \forall K>K_0\ \forall A,B\in\mathcal{M}^+$ $$\varphi(A^2+B^2 +K(AB+BA))\geq0.$$

Usually, the inequalities that characterize the trace are binary, they give a "yes" or "no" type answer, which concludes necessary and sufficient conditions for the element to be or not to be tracial (or central).  However, the latter condition in \cite{AB2019} seems to give different variations of "closeness" to the tracial function. We worked with this conception and obtained a "measure of centrality" and this article as a result.

As we see further the "measure of centrality" in the case of positive bounded operators has a tight connection with the concept of invertibility and the classical results as Gelfand–Mazur theorem and Dixmier characterization of positive operators. In addition, invertibility has already appeared as a property in some sense, similar to commutativity. \cite{Nov2017, Nov2019, Nov2020}

Surprisingly for us, after we had written the first version of the text we discovered that the "centrality measure" that we have defined in this paper from purely algebraic motivation, eventually coincided in its form with the so-called Michelson contrast (also called visibility measure), that is widely applied in image processing \cite{Michelson1993, Michelson2006}, optics \cite{Michelson1927}, vision \cite{Peli2013, Wiebel} and quantum computing \cite{QUA2021}.

\section{Notations and Preliminaries}

We adhere to the following notation. Let $\mathcal{A}$ denote some Banach $\ast$-algebras;
then, $\mathcal{A}^\mathrm{sa},\mathcal{A}^+$ are its self-adjoint and positive parts. $\mathcal{A}^*$ is the conjugate space of continuous linear functionals. If $\mathcal{A}$ is von Neumann algebra, then $\mathcal{A}_*$ denotes its predual space. Additionally, $\mathcal{A}_*^+$, $\mathcal{A}^{*+}$ are the positive cones in $\mathcal{A}_*$ and $\mathcal{A}^*$, respectively.  $\mathrm{Tr}$ denotes the canonical trace of $\mathbb{M}_n(\mathbb{C})$. By $C(H)$ and $B(H)$ we denote the ideal of compact operators and the algebra of bounded operators, respectively.

The following lemma is known at least since 1969 \cite[Lemma 1.6.2]{Dixmier}
\begin{lemma}[Dixmier]\label{1969lemma}
Let $\mathcal{A}$ be a unital $C^*$-algebra with unit $\mathbf{1}$ and its element $x(\in \mathcal{A})$ be hermitian (i.e. $x=x^*$), then
$$\left\|\|x\|\mathbf{1}-x\right\|\leq \|x\| \text{ if and only if } x \text{ is positive.} $$\end{lemma}

%\begin{proposition}[\cite{BikAbed}]
%For any real number $K>0$ there exists positive %functional $\varphi$ on $\mathbb{M}_2(\mathbb{C})$ %such that
%\begin{enumerate}[1)]
% \item $\varphi(A^2+B^2+K(AB+BA))\geq 0$ for any %$A,B\in\mathbb{M}_2^+(\mathbb{C});$
% \item $\varphi\neq\lambda\mathrm{Tr}$ for any %$\lambda>0.$
%\end{enumerate}
%\end{proposition}

%\begin{theorem}[\cite{BikAbed}]
%For $\varphi\in \mathcal{A}_*^+$ the following %conditions are equivalent:
%\begin{enumerate}[1)]
% \item $\varphi$ is tracial;
% \item $\forall K\geq-1\ \forall %A,B\in\mathcal{A}^+$ %$\varphi(A^2+B^2+K(AB+BA))\geq0;$
% \item $\exists K_0>-1\ \forall K>K_0\ \forall %A,B\in\mathcal{A}^+$ %$$\varphi\left(A^2+B^2+K(AB+BA)\right)\geq0.$$
%\end{enumerate}
%\end{theorem}

\section{Centrality and invertibility in $B(H)^+$}\label{sec3}

Let $\mathcal{A}=B(H)$, then the center $\mathfrak{C}(B(H))$ of $B(H)$ is equal to $\mathbb{C}\mathbf{1}$. Let us consider the function
$$\Delta(x)=\inf\limits_{A\in\mathbb{R}^+}\left\{\left\|\mathbf{1}-\frac{x}{A}\right\|\right\} \text{ for } x\in B(H)^+,$$
which illustrates how far the element $x$ is from the central elements.

We see, that if $x=\mathbf{1}$, then $\Delta(\mathbf{1})=0$ ($A=1$) and $\Delta(\mathbf{0})=1.$

\begin{proposition}\label{prop1} Let $x$ be positive operator $(x\in B(H)^+)$, then
$\Delta(x)\leq 1.$
\end{proposition}
\begin{proof}
Straight from the Lemma \ref{1969lemma} we got that $\left\|\mathbf{1}-\frac{x}{\|x\|}\right\|\leq 1$ for any $x\in B(H)^+$, thus the inequalities
$$\inf\limits_{A\in\mathbb{R}^+}\left\|\mathbf{1}-\frac{x}{A}\right\|\leq \left\|\mathbf{1}-\frac{x}{\|x\|}\right\|\leq 1$$
hold.\end{proof}

\begin{proposition}\label{prop2} Let $x$ be positive non-invertible (singular) operator, then
$$\Delta(x)= 1.$$
\end{proposition}
\begin{proof}
We know, that $x$ has the spectrum $$\sigma\left(\frac{x}{A}\right)\subset\left[0,\frac{\|x\|}{A}\right],\text{ where } \inf\sigma(x)=0,\sup\sigma(x)=\|x\|.$$
Thus,
$$\sigma\left(1-\frac{x}{A}\right)\subset\left[1-\frac{\|x\|}{A},1\right].$$
Note, that $1-\frac{x}{A}$ always is selfadjoint, thus
$\left\|1-\frac{x}{A}\right\|\geq 1$ for any $A>0$
and $\inf\limits_{A\in\mathbb{R}^+}\left\|1-\frac{x}{A}\right\|\geq\inf_{A\in\mathbb{R}^+}1 = 1.$
 \end{proof}

\begin{corollary}
Let $x$ be positive compact operator, then $\Delta(x)= 1.$
\end{corollary}

The corollary characterizes compact operators as the kind of most non-central positive elements in the sense of the function $\Delta$ in $B(H)$.

\begin{theorem}\label{Theorem_formula}
Let $x$ be invertible positive operator $(x\in B(H)^+)$, with the inverse $x^{-1}$, then
\begin{equation}\label{inv_formula}
 \Delta(x)=\frac{\|x\|\|x^{-1}\|-1}{\|x\|\|x^{-1}\|+1}.
\end{equation}
\end{theorem}
\begin{proof}
Let now $x$ be invertible element of $B(H)$, then the spectrum $$\sigma(x)\subset\left[\frac{1}{\|x^{-1}\|},\left\|x\right\|\right] \text{ and} \inf \sigma(x) = \frac{1}{\|x^{-1}\|}, \sup\sigma(x)=\|x\|,$$
thus
$$\sigma\left(1-\frac{x}{A}\right)\subset\left[1-\frac{\|x\|}{A},1-\frac{1}{A\|x^{-1}\|}\right],$$

$$\left\| 1-\frac{x}{A}\right\|=\max\left(\left|1-\frac{1}{A\|x^{-1}\|}\right|, \left|1-\frac{\|x\|}{A}\right|\right).$$
Let $\alpha=\min(\frac{1}{\|x^{-1}\|},\|x\|), \beta=\max(\frac{1}{\|x^{-1}\|},\|x\|), \lambda = \frac{1}{A}\geq 0.$
Then the latter maximum is rewritten as
$$\max(|1-\alpha\lambda|,|1-\beta\lambda|)=\frac{|1-\alpha\lambda|+|1-\beta\lambda|+\left||1-\alpha\lambda|-|1-\beta\lambda|\right|}{2}.$$
and we need to minimize it due to $\lambda \in (0,+\infty)$
$$\max(|1-\alpha\lambda|,|1-\beta\lambda|)\to\min.$$

We have here three basic situations, for $\lambda\in(0,\frac{1}{\beta})$, $\lambda \in (\frac{1}{\beta},\frac{1}{\alpha})$ and $\lambda >\frac{1}{\alpha}.$

1) Let $\lambda \in(0,\frac{1}{\beta})$, then
$$|1-\alpha\lambda|+|1-\beta\lambda|+\left||1-\alpha\lambda|-|1-\beta\lambda|\right|=2-2\alpha\lambda\geq 2\frac{\beta-\alpha}{\beta}.$$

2) Let $\lambda>\frac{1}{\alpha}$, then
$$|1-\alpha\lambda|+|1-\beta\lambda|+\left||1-\alpha\lambda|-|1-\beta\lambda|\right|=2\beta\lambda-2\geq2\frac{\beta-\alpha}{\alpha}.$$

3) Let $\lambda \in \left[\frac{1}{\beta},\frac{1}{\alpha}\right],$ then
$$|1-\alpha\lambda|+|1-\beta\lambda|+\left||1-\alpha\lambda|-|1-\beta\lambda|\right|=(\beta-\alpha)\lambda+\left|2-(\alpha+\beta)\lambda\right|.$$
Now either $\lambda>\frac{2}{\alpha+\beta}$ and then
$$(\beta-\alpha)\lambda+(\alpha+\beta)\lambda-2=2(\beta\lambda-1)\geq 2\frac{\beta-\alpha}{\alpha+\beta},$$
or $\lambda<\frac{2}{\alpha+\beta}$ and
$$(\beta-\alpha)\lambda+(\alpha+\beta)\lambda-2=2(1-\alpha\lambda)\geq 2\frac{\beta-\alpha}{\alpha+\beta}.$$

At last, $$\inf\{\max(|1-\alpha\lambda|,|1-\beta\lambda|)\}=\frac{(\beta-\alpha)}{\max\{\alpha,\beta,\alpha+\beta\}}=\frac{\beta-\alpha}{\alpha+\beta}.$$

Now, note, that $$\beta-\alpha=\max\left(\frac{1}{\|x^{-1}\|},\|x\|\right)-\min\left(\frac{1}{\|x^{-1}\|},\|x\|\right)=$$
$$=
\left|\|x\|-\frac{1}{\|x^{-1}\|}\right|=\left|\frac{\|x\||x^{-1}\|-1}{\|x^{-1}\|}\right|.$$
Also, note, that $1=\|\mathbf{1}\|=\|xx^{-1}\|\leq\|x\|\|x^{-1}\|$,
thus $$\beta-\alpha=\frac{\|x\||x^{-1}\|-1}{\|x^{-1}\|}.$$
To accomplish the proof, note, that $$\alpha+\beta=\frac{\|x\|\|x^{-1}\|+1}{\|x^{-1}\|}.$$\end{proof}

\begin{remark}
Note, that from the latter theorem it follows, that for any $x\in B(H)^+$ the equality $\Delta(x)=\Delta(x^{-1})$ holds.
\end{remark}

\begin{remark}
Note, that $\Delta(x)=0$ for the invertible operator, only if $\|x^{-1}\|=\frac{1}{\|x\|},$ which would mean $\sigma(x)=[\|x\|,\|x\|],$
thus $x=\|x\|\mathbf{1}.$
\end{remark}

\begin{corollary}
Let $x\in B(H)^+$ be invertible element with the inverse element $x^{-1}\in B(H)$, then $\Delta(x)<1$.
\end{corollary}

\begin{corollary}\label{continuous}
Let the sequence $x_n$ from $B(H)^+$ that converges to element $\mathbf{0}\neq x\in B(H)^+$ in terms of norm,  $$\lim_n\Delta(x_n)=\Delta(x),$$
i.e. $\Delta: (B(H)^+\setminus\{\mathbf{0}\},\|\cdot\|)\mapsto [0,1]$ is a continuous function.
\end{corollary}
\begin{proof}
It is sufficient to consider four cases:
\begin{enumerate}[1)]
 \item $x_n$ consists of invertible elements and $x$ is invertible;
 \item $x_n$ consists of invertible elements and $x$ is not invertible;
 \item $x_n$ consists of non-invertible elements and $x$ is invertible;
 \item $x_n$ consists of non-invertible elements and $x$ is not invertible.
\end{enumerate}
The cases $1)$ and $4)$ are evident.

Consider $2)$, when $x_n$ are invertible and $x$ is not. Then
$$\Delta(x_n)=\frac{\|x^{-1}_n\|-\frac{1}{\|x_n\|}}{\|x^{-1}_n\|+\frac{1}{\|x_n\|}}.$$
The sequence $\|x_n^{-1}\|$ converges to infinity.
Note, that $$\|x_n^{-1}-x_m^{-1}\|\leq \|x_n^{-1}\|\|x_m^{-1}\|\|x_n-x_m\|,$$ thus if $\|x_n^{-1}\|$ has a bounded sub-sequence $\|x_{n_k}\|$, then $x_{n_k}^{-1}$ is fundamental, therefore there exists its limit $z$ such that
$x_{n_k}^{-1}x_{n_k}\to z x$ and
$x_{n_k}x_{n_k}^{-1}\to x z$ on one hand, and
$x_{n_k}^{-1}x_{n_k}\to \mathbf{1}$ on the other hand, thus $z=x^{-1}$, which is forbidden by assumption. Therefore, any sub-sequence of $\|x_n^{-1}\|$ is unbounded.

Since $\|x_n^{-1}\|\to \infty$ and $\|x_n\|\to \|x\|$, it follows, that $\Delta(x_n)\to 1.$

Consider $3)$ with $x_n$ being a sequence of non-invertible positive operators converging to invertible operator $x$ by the norm. Consider $x_n'=\frac{1}{n}+x_n$, then each of $x_n'$ is invertible and $x_n'\to x$, thus
$\Delta(x_n')\to\Delta(x).$ On the other hand, by the formula \ref{inv_formula}
$$|\Delta(x_n)-\Delta(x_n')|=
\left|1-\frac{\|\frac{1}{n}+x_n\|\|(\frac{1}{n}+x_n)^{-1}\|-1}{\|\frac{1}{n}+x_n\|\|(\frac{1}{n}+x_n)^{-1}\|+1}\right|=$$
$$=\frac{2}{\|\frac{1}{n}+x_n\|\|(\frac{1}{n}+x_n)^{-1}\|+1}=\frac{2}{2+n\|x_n\|}\to 0,$$
thus $|\Delta(x_n)-\Delta(x)|\leq|\Delta(x_n)-\Delta(x_n')|+|\Delta(x_n')-\Delta(x)|\to 0.$
 \end{proof}
From the latter two corollaries, we obtain the following.

\begin{corollary}
If the sequence of operators is converging to a non-singular (invertible) operator, then the sequence contains not more than a finite quantity of non-invertible operators.
\end{corollary}

\begin{corollary}
For any $x$ in $B(H)^+$ the following properties
\begin{enumerate}[1)]
 \item $$\Delta(x)=\frac{\sup(\sigma(x))-\inf(\sigma(x))}{\sup(\sigma(x))+\inf(\sigma(x))};$$
 \item $$\Delta(x)=\frac{\sup\sigma(x)}{\sup\sigma(x)+\inf\sigma(x)}\left\|1-\frac{x}{\|x\|}\right\|.$$
\end{enumerate}
hold.
\end{corollary}
\begin{proof}
$1)$ follows straight from the proof of the theorem.
$2)$ Since invertible elements are dense in $B(H)$ and $\Delta$ is continuous, it is sufficient to consider invertible $x$. By Theorem \ref{Theorem_formula} we have
$$\Delta(x)=\frac{\|x\|\|x^{-1}\|-1}{\|x\|\|x^{-1}\|+1}.$$
At the same time, $\mathbf{1}-\frac{x}{\|x\|}\geq \mathbf{0},$ thus $$\left\|1-\frac{x}{\|x\|}\right\|=\sup\sigma(1-\frac{x}{\|x\|})=1-\inf\sigma\left(\frac{x}{\|x\|}\right)=$$
$$=1-\frac{1}{\|x\|}\inf\sigma (x)=1-\frac{1}{\|x\|}\frac{1}{\|x^{-1}\|}=\frac{\|x\|\|x^{-1}\|-1}{\|x\|\|x^{-1}\|}.$$

At last,
$$\Delta(x)=\frac{\|x\|\|x^{-1}\|}{\|x\|\|x^{-1}\|+1}\frac{\|x\|\|x^{-1}\|-1}{\|x\|\|x^{-1}\|}=
\frac{\|x\|}{\|x\|+\frac{1}{\|x^{-1}\|}}\left\|\mathbf{1}-\frac{x}{\|x\|}\right\|$$
 \end{proof}

\begin{corollary}
Let $x\in B(H)^+$, then
\begin{enumerate}[1)]
 \item $$\frac{1}{2}\left\|1-\frac{x}{\|x\|}\right\|\leq \Delta(x)\leq \left\|1-\frac{x}{\|x\|}\right\|;$$
 \item $\Delta(x)=0$ only if $x=\|x\|\mathbf{1}.$
\end{enumerate}
\end{corollary}

\begin{remark}
Note, that the algebra of compact operators $C(H)$ is dense in $B(H)$ in weak operator topology (and $C(H)^+$ in $B(H)^+$, respectively), thus $\Delta$ cannot be continuous in weak operator topology.
\end{remark}

Now let us describe some properties of $\Delta.$
\begin{theorem}\label{equpotential}
Let $x,y\in B(H)^+$, $\lambda\in\mathbb{R}^+$ then
\begin{enumerate}[1)]
 \item $\Delta(\lambda x)=\Delta(x);$
 \item if $x, y$ are invertible, then
 $$\left(\|x+y\|+\frac{1}{\|(x+y)^{-1}\|}\right)\Delta(x+y)\leq$$
 $$\leq\left(\|x\|+\frac{1}{\|x^{-1}\|}\right)\Delta(x)+\left(\|y\|+\frac{1}{\|y^{-1}\|}\right)\Delta(y).$$
 The inequality becomes an equality if $x=\lambda y$.
% \item $\Delta(x+y)\leq \Delta(x)+\Delta(y).$

\end{enumerate}
\end{theorem}
\begin{proof}
$1)$ By definition
$$\Delta(\lambda x)=\inf\limits_{A\in\mathbb{R}^+}\left\{\left\|\mathbf{1}-\frac{\lambda x}{A}\right\|\right\}=\inf\limits_{A'\in\mathbb{R}^+, A'=A/\lambda}\left\{\left\|\mathbf{1}-\frac{x}{A'}\right\|\right\}=\Delta(x).$$

$2)$ Note, that $x+y$ is invertible, $\|x+y\|\mathbf{1}\leq \|x\|\mathbf{1}+\|y\|\mathbf{1}$, thus
$$0\leq \|x+y\|\mathbf{1}-(x+y)\leq \|x\|\mathbf{1}-x+\|y\|\mathbf{1}-y,$$
therefore
$$\|\|x+y\|\mathbf{1}-(x+y)\|\leq \|\|x\|\mathbf{1}-x\|+\|\|y\|\mathbf{1}-y\|.$$
In terms of upper and lower bounds, we see that
$$\sup\sigma(x+y)-\inf\sigma(x+y)\leq \sup\sigma(x)-\inf\sigma(x)+\sup\sigma(y)-\inf\sigma(y),$$
thus
$$\Delta(x+y)=\frac{\sup\sigma(x+y)-\inf\sigma(x+y)}{\sup\sigma(x+y)+\inf\sigma(x+y)}\leq$$
$$\leq \frac{\sup\sigma(x)-\inf\sigma(x)}{\sup\sigma(x+y)+\inf\sigma(x+y)}+\frac{\sup\sigma(y)-\inf\sigma(y)}{\sup\sigma(x+y)+\inf\sigma(x+y)}=$$
$$=\frac{\sup\sigma(x)+\inf\sigma(x)}{\sup\sigma(x+y)+\inf\sigma(x+y)}\Delta(x)+\frac{\sup\sigma(y)+\inf\sigma(y)}{\sup\sigma(x+y)+\inf\sigma(x+y)}\Delta(y).$$

Note, that if $x=\lambda y$, then
$$\left(\|x+y\|+\frac{1}{\|(x+y)^{-1}\|}\right)\Delta(x+y)=$$
$$=(1+\lambda)\left(\|x\|+\frac{1}{\|x^{-1}\|}\right)\Delta\left((1+\lambda)x\right)=$$
$$=\left(\|x\|+\frac{1}{\|x^{-1}\|}\right)\Delta(x)+\lambda\left(\|x\|+\frac{1}{\|x^{-1}\|}\right)\Delta(x)=$$
$$=\left(\|x\|+\frac{1}{\|x^{-1}\|}\right)\Delta(x)+\left(\|\lambda x\|+\frac{1}{\|(\lambda x)^{-1}\|}\right)\Delta(\lambda x).$$

%$3)$ Now, note, that since $\mathbf{0}\leq x\leq %x+y$
%we have the inequalities $\sup\sigma(x)\leq %\sup\sigma(x+y)$ and $\inf\sigma(x)\leq %\inf\sigma(x+y)$, thus
%$$\frac{\sup\sigma(x)+\inf\sigma(x)}{\sup\sigma(x+%y)+\inf\sigma(x+y)}\leq 1,$$
%the same inequality holds if we swap $x$ and $y$.
%
\end{proof}

\begin{corollary}
Let $x,y\in B(H)^+$, then from inequality $iv)$ of the latter Theorem we obtain the following
\begin{enumerate}[1)]
 \item if $x$,$y$ are singular (non-invertible) and $x+y$ is invertible (non-singular), then $$\left(\|x+y\|+\frac{1}{\|(x+y)^{-1}\|}\right)\Delta(x+y)\leq \|x\|\Delta(x)+\|y\|\Delta(y);$$
 \item if $x+y$ is singular (non-invertible), then $$\|x+y\|\Delta(x+y)\leq \|x\|\Delta(x)+\|y\|\Delta(y);$$
 \item if $x$ is invertible and $y$ is singular (non-invertible), then $$\left(\|x+y\|+\frac{1}{\|(x+y)^{-1}\|}\right)\Delta(x+y)\leq$$
 $$\leq\left(\|x\|+\frac{1}{\|x^{-1}\|}\right)\Delta(x)+\|y\|\Delta(y).$$
\end{enumerate}
\end{corollary}

\begin{lemma}\label{lemma}
Let $X,Y\in \mathbb{M}_2(\mathbb{C})^+$, then
 $\Delta( X+ Y)\leq \max(\Delta(X),\Delta(Y)).$
\end{lemma}
\begin{proof}
Using Jordan normal form decomposition we
assume, that
$$
 X = \left(\begin{array}{cc}
 \alpha_1 & 0 \\
 0 & \beta_1 \\
 \end{array}\right)
\text{ and }
Y=
 \left(\begin{array}{cc}
 \beta_2+(\alpha_2-\beta_2)\lambda & (\alpha_2-\beta_2)\delta\sqrt{\lambda(1-\lambda)} \\
 (\alpha_2-\beta_2)\overline{\delta}\sqrt{\lambda(1-\lambda)} & \alpha_2+(\beta_2-\alpha_2)\lambda \\
 \end{array}\right);$$
where $\alpha_{1,2}, \beta_{1,2}\geq 0$, $\lambda\in[0,1]$ and $|\delta|=1$, $\delta\in \mathbb{C}$.

Evidently, $$\Delta(X)=\frac{|\beta_1-\alpha_1|}{\beta_1+\alpha_1} \text{ and } \Delta(Y)=\frac{|\beta_2-\alpha_2|}{\beta_2+\alpha_2}.$$

Let us find $\Delta(X+Y)$ by finding the eigenvalues of $X+Y$. We have
$$\left(\alpha_1+\alpha_2\lambda+\beta_2-\beta_2\lambda-t\right)\times$$
$$\times\left(\beta_1+\alpha_2-\alpha_2\lambda+\beta_2\lambda-t\right)-(\alpha_2-\beta_2)^2\lambda(1-\lambda)=0$$
We expand the first two brackets and regroup them into the quadratic equation
$$t^2-t\left((\alpha_2-\beta_2)\lambda+(\alpha_1+\beta_2)+(\beta_2-\alpha_2)\lambda+(\beta_1+\alpha_2)\right)-$$
$$-(\beta_2-\alpha_2)^2\lambda^2+(\alpha_1+\beta_2)(\beta_2-\alpha_2)\lambda+$$
$$+(\alpha_2+\beta_1)(\alpha_2-\beta_2)\lambda+(\alpha_1+\beta_2)(\beta_1+\alpha_2)-$$
$$-(\alpha_2-\beta_2)^2\lambda+(\alpha_2-\beta_2)^2\lambda^2=0.$$
We simplify the equation and get the following
$$t^2-t\left(\alpha_1+\alpha_2+\beta_2+\beta_1\right)+(\alpha_1-\beta_1+\beta_2-\alpha_2)(\beta_2-\alpha_2)\lambda+$$
$$+(\alpha_1+\beta_2)(\beta_1+\alpha_2)-(\beta_2-\alpha_2)^2\lambda=0.$$
At last, we obtain the equation
$$t^2-t\left(\alpha_1+\alpha_2+\beta_2+\beta_1\right)-(\beta_1-\alpha_1)(\beta_2-\alpha_2)\lambda+(\alpha_1+\beta_2)(\beta_1+\alpha_2)=0.$$

Note, that
$$\Delta(X+Y)=\frac{|t_1-t_2|}{t_1+t_2}=\frac{\sqrt{D}}{2b},$$
where $D=b^2-4ac$, $$a=1, b=\alpha_1+\alpha_2+\beta_1+\beta_2,$$ $$c=\left((\beta_1-\alpha_1)(\beta_2-\alpha_2)\lambda-(\alpha_1+\beta_2)(\beta_1+\alpha_2)\right).$$

Assume $\Delta(X)\leq\Delta(Y)$, that will give us the condition $|(b_1-a_1)(b_2+a_2)|\leq |(b_2-a_2)(b_1+a_1)|.$ Without loss of generality we can assume, that $a_1\leq b_1$ and $a_2\leq b_2$ (if it is not, then we can use renaming and changing $\lambda$ to $1-\lambda$), then
$0\leq b_2a_1-a_2b_1$.

We need to prove, that
$$\mathrm{Tr}^2(X+Y)+4\mathrm{Tr}(X)\mathrm{Tr}(Y)\Delta(X)\Delta(Y)\lambda\leq 4(\alpha_1+\beta_2)(\beta_1+\alpha_2)+\Delta(Y)^2\mathrm{Tr}^2(X+Y).$$
We divide the inequality by $\mathrm{Tr}^2(X+Y)$
and get

$$1+4\lambda\frac{\mathrm{Tr}(X)\mathrm{Tr}(Y)}{\left(\mathrm{Tr}(X)+\mathrm{Tr}(Y)\right)^2}\Delta(X)\Delta(Y)\leq 4(1-\theta)\theta+\Delta(Y)^2,$$
where $$\theta_1=\frac{\beta_1+\alpha_2}{\mathrm{Tr}(X+Y)} \text{ and } \theta_2=4\frac{\mathrm{Tr}(X)\mathrm{Tr}(Y)}{\left(\mathrm{Tr}(X)+\mathrm{Tr}(Y)\right)^2}.$$ Evidently, $\theta_1,\theta_2\in[0,1].$
We rewrite the inequality in the form
$$\lambda\theta_2\Delta(X)\Delta(Y)\leq \Delta(Y)^2-(1-2\theta_1)^2.$$
The latter inequality holds if and only if
$\Delta(X+Y)\leq\max(\Delta(X),\Delta(Y))$.
We see that $\lambda$ affects only the left-hand side of inequality, thus if the inequality does not hold, then there exists a counter-example in the form

\centerline{
\begin{tabular}{cc}
 $X = \left(\begin{array}{cc}
 \alpha_1 & 0 \\
 0 & \beta_1 \\
 \end{array}\right)$;& $Y=\left(
 \begin{array}{cc}
 \alpha_2 & 0 \\
 0 & \beta_2 \\
 \end{array}\right);$
\end{tabular}
with $\alpha_1\leq\beta_1$, $\alpha_2\leq\beta_2$.}
\noindent Then $$\Delta(X+Y)=\frac{\beta_2+\beta_1-\alpha_2-\alpha_1}{\alpha_1+\alpha_2+\beta_1+\beta_2}>\frac{\beta_2-\alpha_2}{\beta_2+\alpha_2}=\Delta(Y)\geq\Delta(X)=\frac{\beta_1-\alpha_1}{\beta_1+\alpha_1},$$
so $$(\beta_2-\alpha_2)(\beta_2+\alpha_2)+(\beta_1-\alpha_1)(\beta_2+\alpha_2)>$$
$$>(\beta_2-\alpha_2)(\beta_2+\alpha_2)+(\beta_2-\alpha_2)(\beta_1+\alpha_1),$$
and
$$(\beta_2-\alpha_2)(\beta_1+\alpha_1)\leq (\beta_2+\alpha_2)(\beta_1-\alpha_1)$$
at the same time, which cannot be true.
 \end{proof}

\begin{lemma}\label{lemma3}
Let $X,Y\in \mathbb{M}_n(\mathbb{C})^+$, then $\Delta(X+Y)\leq \max(\Delta(X),\Delta(Y)).$
\end{lemma}
\begin{proof}
Note, that in finite-dimensional spaces the unit ball is a compact set, thus

$$\sup_{\|f\|=1}{\langle (X+Y) f, f\rangle}=\langle (X+Y)f_0, f_0\rangle;$$
$$\sup_{\|g\|=1}{\langle (X+Y) g, g\rangle}=\langle (X+Y)g_0, g_0\rangle$$
for some $f_0, g_0\in \mathbb{C}^n$, $\|f_0\|=\|g_0\|=1$. By $H_2$ we denote the complex linear span $\mathrm{Lin}_\mathbb{C}\{f_0,g_0\}$ of these two vectors.

Evidently, for the restriction $(X+Y)|_{H_2}$
we have $(X+Y)|_{H_2}f_0=(X+Y)f_0$, $(X+Y)|_{H_2}g_0=(X+Y)g_0$ and, also,
$$\langle (X+Y)|_{H_2} f_0, f_0\rangle\leq \sup_{\|f\|=1, f\in H_2}(\langle (X+Y)|_{H_2} f, f\rangle)\leq$$
$$\leq \sup_{\|f\|=1, f\in \mathbb{C}^n}(\langle (X+Y) f, f\rangle)=\langle (X+Y) f_0, f_0\rangle;$$
$$\langle (X+Y)|_{H_2} f_0, f_0\rangle\geq \inf_{\|f\|=1, f\in H_2}(\langle (X+Y)|_{H_2} f, f\rangle)\geq$$
$$\geq \inf_{\|f\|=1, f\in \mathbb{C}^n}(\langle (X+Y) f, f\rangle)=\langle (X+Y) f_0, f_0\rangle.$$

Now, from the latter inequalities we gain the fact, that $$\Delta(X+Y)=\Delta((X+Y)|_{H_2})=\Delta(X|_{H_2}+Y|_{H_2}).$$

Also, note, that $$0\leq \sup_{\|f\|=1, f\in {H_2}}(X|_{H_2}f,f)\leq \sup_{\|f\|=1,f\in {\mathbb{C}^n}}(Xf,f);$$
$$0\leq \inf_{\|f\|=1,f\in {\mathbb{C}^n}}(X f,f)\leq \inf_{\|f\|=1,f\in {H_2}}(X|_{H_2}f,f).$$

Now, consider two cases:

1) If $\inf_{\|f\|=1,f\in {\mathbb{C}^n}}(X f,f)=0$, then $\Delta(X|_{H_2})\leq 1 = \Delta(X)$

2) Otherwise, $\inf_{\|f\|=1,f\in {\mathbb{C}^n}}(X f,f)\neq0$ and we denote
$$t=\frac{\sup_{\|f\|=1, f\in {H_2}}(X|_{H_2}f,f)}{\inf_{\|f\|=1,f\in {H_2}}(X|_{H_2}f,f)}; T=\frac{\sup_{\|f\|=1,f\in {\mathbb{C}^n}}(Xf,f)}{\inf_{\|f\|=1,f\in {\mathbb{C}^n}}(X f,f)}.$$
$t\leq T$. So now we have
$$\Delta(X_{H_2})=\frac{\sup_{f\in H_2, \|f\|=1}(\langle X|_{H_2}f,f\rangle-\inf_{f\in H_2, \|f\|=1}(\langle X|_{H_2}f,f\rangle}{\sup_{f\in H_2, \|f\|=1}(\langle X|_{H_2}f,f\rangle+\inf_{f\in H_2, \|f\|=1}(\langle X|_{H_2}f,f\rangle}=$$
$$=\frac{t-1}{t+1}=1-\frac{2}{t+1}\leq 1-\frac{2}{T+1}=$$
$$=\frac{\sup_{f\in \mathbb{C}^n, \|f\|=1}(\langle X f,f\rangle-\inf_{f\in \mathbb{C}^n, \|f\|=1}(\langle Xf,f\rangle}{\sup_{f\in \mathbb{C}^n, \|f\|=1}(\langle Xf,f\rangle+\inf_{f\in \mathbb{C}^n, \|f\|=1}(\langle Xf,f\rangle}=\Delta(X).
$$

By the Lemma \ref{lemma} we have, that
$$\Delta(X|_{H_2}+Y|_{H_2})\leq \max(\Delta(X|_{H_2}),\Delta(Y|_{H_2}))$$
and we just have proved that $$\Delta(X+Y)=\Delta(X|_{H_2}+Y|_{H_2}); \max(\Delta(X|_{H_2}),\Delta(Y|_{H_2}))\leq \max(\Delta(X),\Delta(Y)).$$
\end{proof}
%\begin{corollary}
%For any two finite-rank positive operators $x,y\in %F(H)^+\subset B(H)^+$ the inequality %$$\Delta(x+y)\leq \max\{\Delta(x),\Delta(y)\}$$ %holds.
%\end{corollary}
%\begin{proof}
%Evidently, if we take finite-dimensional positive %(and thus self-adjoint and bounded) opearator $x$, %then we have that $H=\mathrm{rg\ } x\oplus\ker x$ %with $\dim\mathrm{rg\ } x <+\infty$. Thus, if we %take the complex linear span %$\mathcal{H}=\mathrm{Lin}_\mathbb{C}\{\mathrm{rg\ } %x, \mathrm{rg\ } y \}$ of $\mathrm{rg } x$ and %$\mathrm{rg } y$, then $\mathcal{H}$ is finite %dimensional $\dim \mathcal{H}<+\infty$,
%we have one-to-one correspondense of %$x|_\mathcal{H}, y|_\mathcal{H}$ and $x,y$ %correspondingly and $x|_\mathcal{H}: \mathcal{H} %\mapsto \mathcal{H}$ and $x|_\mathcal{H}: %\mathcal{H}\mapsto \mathcal{H},$ thus they have a %representation in the form of $X,Y\in %\mathbb{M}_{\mathrm{dim}\mathcal{H}}(\mathbb{C})^+.%$
%\end{proof}
\begin{theorem}\label{Theorem4}
Let $x,y\in B(H)^+$, then $\Delta(x+y)\leq\max\{ \Delta(x),\Delta(y)\}.$
\end{theorem}
\begin{proof}
Let $\{f_i\}_{i=1}^\infty$ be the orthonormal basis of $H$, then let $$p_n=\sum\limits_{i=1}^\infty\langle \cdot f_i,f_i\rangle f_i$$
be the projection to the linear span of $\{f_i\}_{i=\overline{1,n}}.$
By $\Delta_n(x)$ we denote $\Delta(x|_{p_n})$ in the reduced algebra $B(H)_{p_n}$ of operators on $p_nHp_n$ hilbert space, i.e.
$$\Delta_n(x)=\Delta(x_{p_n}).$$
Then $\Delta_n(x+y)\leq \max(\Delta_n(x),\Delta_n(y))$ by Lemma \ref{lemma3}.

Now, note that for arbitrary operator $a$ $\sup(\sigma(a_{p_n}))\leq\sup(\sigma(a_{p_{n+1}}))\leq \sup(\sigma(a))$
and $\inf(\sigma(a_{p_n}))\geq \inf(\sigma(a_{p_{n+1}}))\geq \inf(\sigma(a))$, and, also, $$\lim_{n\to\infty}\sup(\sigma(a_{p_n}))=\sup(\sigma(a));$$
$$\lim_{n\to\infty}\inf(\sigma(a_{p_n}))=\inf(\sigma(a)).$$

Now, if $\inf(\sigma(x))=0$, then $\Delta(x)=1$
and $\Delta(x+y)\leq 1 =\max(\Delta(x),\Delta(y))$.
Assume, that $\inf(\sigma(x))\neq0$, $\inf(\sigma(y))\neq 0$ and thus $\inf(\sigma(x+y))\neq0$. For $a$ such that $\inf\sigma(a)\neq 0$
by $t_n^a$  we denote $$t_n^a:=\frac{\sup(\sigma(a_{p_n}))}{\inf(\sigma(a_{p_n}))}; \text{ note that } t_n^a\xrightarrow{n}t^a=\frac{\sup(\sigma(a))}{\inf(\sigma(a))}.$$
Then $$\lim_{n\to\infty}\Delta_n(x)=\lim_{n\to\infty}\frac{t_n^x-1}{t_n^x+1}=\Delta(x)$$
$$\lim_{n\to\infty}\Delta_n(y)=\lim_{n\to\infty}\frac{t_n^y-1}{t_n^y+1}=\Delta(y)$$
$$\lim_{n\to\infty}\Delta_n(x+y)=\lim_{n\to\infty}\frac{t_n^{x+y}-1}{t_n^{x+y}+1}=\Delta(x+y),$$
and since $\Delta_n(x+y)\leq\max\{\Delta_n(x),\Delta_n(y)\},$ we get the proposition of the theorem.
\end{proof}

One can ask when equality holds. We do not have a comprehensive answer for that, but we assume that the following example would be interesting under these considerations.

\begin{example}
Let $X,Y\in\mathbb{M}_2(\mathbb{C})^+$ and commute i.e. $XY=YX$, then the equality
$\Delta(X+Y)=\max(\Delta(X),\Delta(Y))$ holds only if $X=\lambda Y,$ where $\lambda \in \mathbb{R}^+.$

Indeed, note that $X=a_xp+b_xp^\perp$ and $Y=a_yp+b_yp^\perp$, where $p$ is non-zero and non-unit projection $\mathbb{M}_2(\mathbb{C})^\mathrm{pr}$. Consider
$$\frac{|a_x+a_y-b_x-b_y|}{a_x+b_x+a_y+b_y}=\frac{a_x-b_x}{a_x+b_x}.$$
Evidently, in that case $b_x\neq a_x$ (consider the contrary, then $\max\{\Delta(X),\Delta(Y)\}=\Delta(Y)$).
From the equality we obtain
$$\left|1+\frac{a_y-b_y}{a_x-b_x}\right|=1+\frac{a_y+b_y}{a_x+b_x}.$$
Now, either
$$2+\frac{a_y+b_y}{a_x+b_x}+\frac{a_y-b_y}{a_x-b_x}=0$$
or
$$(a_y-b_y)(a_x+b_x)=(a_y+b_y)(a_x-b_x).$$
In the first case
$(a_y+b_y)(a_x-b_x)+(a_y-b_y)(a_x+b_x)+2(a_x+b_x)(a_x-b_x)=0;$
$$a_ya_x+b_ya_x-a_yb_x-b_xb_y+a_xa_y+{a_yb_x}-a_xb_y-b_xb_y+2a_x^2-2b_x^2=0;$$
$$2a_xa_y-2b_xb_y+2a_x^2-2b_x^2=0;$$
$$a_x(a_y+a_x)=b_x(b_x+b_y);$$
thus $$\frac{a_x}{b_x}=\frac{b_x+b_y}{a_x+a_y}.$$

Without loss of generality we may assume that $$a_x+a_y=1 \text{ and } \frac{a_x}{b_x}=\frac{b_x+b_y}{a_x+a_y}=\lambda<1;$$
then
$$\Delta(X+Y)=\Delta\begin{pmatrix}1 & 0 \\ 0 & \lambda \end{pmatrix}=\frac{1-\lambda}{1+\lambda} \text{ and } \Delta(X)=\Delta\begin{pmatrix}\mu\lambda & 0 \\ 0 & \mu \end{pmatrix}=\frac{\mu-\lambda\mu}{\mu+\lambda\mu}=\frac{1-\lambda}{1+\lambda}.$$
with $0\neq\mu<\lambda$, which leads to contradiction, since  $$\Delta(Y)=\Delta\begin{pmatrix}1-\mu\lambda & 0 \\ 0 & \lambda - \mu\end{pmatrix}=\frac{1-\mu\lambda-\lambda+\mu}{1-\mu\lambda+\lambda-\mu}=\frac{(1-\lambda)(1+\mu)}{(1+\lambda)(1-\mu)}>\frac{1-\lambda}{1+\lambda},$$
which leads to contradiction with the assumption that $\max\{\Delta(X),\Delta(Y)\}=\Delta(X).$

In the other case
$$a_xa_y-a_xb_y+a_yb_x-b_xb_y = a_xa_y-b_xa_y+b_ya_x-b_xb_y;$$
$$a_yb_x=a_xb_y,$$
and at last, $a_x/b_x=a_y/b_y,$ thus if we denote $a_x/b_x=a_y/b_y$ as $\lambda$
$$\frac{a_x}{b_x}=\frac{a_y}{b_y}=\frac{\lambda}{b_x+b_y}(b_x+b_y)=\frac{\lambda b_x}{b_x+b_y}+\frac{\lambda b_y}{b_x+b_y}=$$
$$=\frac{a_x}{b_x+b_y}+\frac{a_y}{b_x+b_y}=\frac{a_x+a_y}{b_x+b_y}.$$

Therefore, if by $\lambda$ we denote $\lambda = a_x/b_x=a_y/b_y$

$$X=\begin{pmatrix} a_x & 0 \\ 0 & b_x \\
\end{pmatrix}; Y=\begin{pmatrix} a_y & 0 \\ 0 & b_y ;\\
\end{pmatrix}$$

$$\Delta(X)=\frac{|a_x-b_x|}{a_x+b_x}=\frac{|\lambda-1|}{\lambda+1};\ \Delta(Y)=\frac{|a_y-b_y|}{a_y+b_y}=\frac{|\lambda-1|}{\lambda+1};$$
$$\Delta(X+Y)=\frac{|a_x+a_y-b_x-b_y|}{a_x+a_y+b_x+b_y}=\frac{|\lambda-1|}{\lambda+1}.$$

\end{example}

\begin{remark}
Note, that $\Delta$ does not hold monotonicity.

For example, take
$x=2p+4p^\perp$ and $y=3p+9p^\perp$ with $(\mathbf{0},\mathbf{1}\neq) p\in B(H)^\mathrm{pr}$, then $x\leq y$, but $$\Delta(x)=\frac{4-2}{4+2}=\frac{1}{3}
\text{ and } \Delta(y)=\frac{9-3}{9+3}=\frac{1}{2}\text{ and } \Delta(x)=\frac{1}{3}\leq\frac{1}{2}=\Delta(y).$$

As another example take $x=p+2p^\perp$ and
$y=2p+3p^\perp$, then $x\leq y$ and
$$\Delta(x)=\frac{2-1}{2+1}=\frac{1}{3}\geq\frac{1}{5}=\frac{3-2}{3+2}=\Delta(y).$$
\end{remark}

\begin{remark}
Note, that the inequality $\min(\Delta(x),\Delta(y))\leq \Delta(x+y)$ does not hold, for example if $(\mathbf{1},\mathbf{0}\neq)x=p\in B(H)^\mathrm{pr}$ and $y=p^\perp$, then $\Delta(p)=\Delta(p^\perp)=1$ and $0=\Delta(\mathbf{1})=\Delta(p+p^\perp).$
\end{remark}

Based on Theorems \ref{equpotential}  and \ref{Theorem4} for $B(H)$ we can define a cone $$K_c=\{x\in B(H)^+ | \Delta(x)\leq c \},$$ which would be a closed subcone with the property, that if $0\leq c_1\leq c_2\leq 1$, then $$\mathbb{R}^+\mathbf{1}\subset K_0\subset K_{c_1}\subset K_{c_2}\subset K_1=\mathcal{B}^+(H).$$

The latter implies, that the cone of positive elements without a zero element has a fibration into the sets $K_c\setminus \bigcup_{0\leq b<c} K_b$ parametrized by interval $[0,1]$.

\begin{theorem}\label{theorem5}
Let $x,y\in B(H)^+$, then the following inequalities
\begin{enumerate}[1)]
 \item if $xy\in  B(H)^+$, then $\Delta(xy)=\Delta(yx)$;
 \item if $xy\in B(H)^+$, then $\Delta(xy)\leq\max(\Delta(x^2),\Delta(y^2));$
 \item $\Delta(x^2)\leq 2\Delta(x);$
 \item $\Delta(x)\leq\Delta(x^2)$
\end{enumerate}
hold.
\end{theorem}
\begin{proof}

$i)$ Note that $\sigma(xy)\setminus\{0\}=\sigma(yx)\setminus\{0\}$.
Assume that $0\notin \sigma(x)$ and $0\notin \sigma(y)$, then $\ker{x}=\ker{y}=\{\overrightarrow{0}\}$, thus $\ker xy=\ker yx= \{\overrightarrow{0}\}$ (without loss of generality, assume $xyf=\mathbf{0}$, then if $yf\neq\overrightarrow{0}$, then $\ker x\neq\{\overrightarrow{0}\}$ (that is impossible), or $yf=\{\overrightarrow{0}\}$, which would mean that $\ker y \neq 0$). Thus, $$\sigma(xy)=\sigma(xy)\setminus\{0\}=\sigma(yx)\setminus\{0\}=\sigma(yx).$$
If $0\in\sigma(x)$ and $0\in \sigma(y)$, evidently, $$\Delta(xy)=1=\Delta(yx)$$.
Now, assume $0\in\sigma(x)$ and $0\notin \sigma(y)$, then $\mathrm{Img} y=H$, thus if $xg=\overrightarrow{0}$ we always have the preimage $y^{-1}g=f$ such that $xyf=0$, thus $0\in \sigma(xy)$ and $\Delta(xy)=1$. Even simplier we obtain that $\Delta(yx)=1$, thus $\Delta(xy)=\Delta(yx)$.

$ii)$ If $x$ or $y$ is singular (non-invertible), then the inequality is evident.

Let $x$ and $y$ be invertible, then
$$\|xy\|\|(xy)^{-1}\|=\|xy\|\|y^{-1}x^{-1}\|\leq
\|x\|\|y\|\|y^{-1}\|\|x^{-1}\|,$$
Assume, that $\Delta(x)\leq \Delta(y)$, then
$$\|x\|\|x^{-1}\|\leq \|y\|\|y^{-1}\|.$$
Now,
$$\Delta(xy)=1-\frac{2}{\|xy\|\|y^{-1}x^{-1}\|+1}\leq  1-\frac{2}{\|x\|\|y\|\|y^{-1}\|\|x^{-1}\|+1}\leq$$ $$\leq 1-\frac{2}{\|y\|^2\|y^{-1}\|^2+1}=1-\frac{2}{\|y^2\|\|y^{-2}\|+1}=\Delta(y^2).$$
We used here the property $\|x^2\|=\|x\|^2$ of a $C^*$-algebra.

$iii)$ We have $$\Delta(x^2)=\frac{\|x^2\|\|x^{-2}\|-1}{\|x^2\|\|x^{-2}\|+1}=\frac{\|x\|\|x^{-1}\|-1}{\|x\|\|x^{-1}\|+1}\times\frac{\left(\|x\|\|x^{-1}\|+1\right)^2}{\|x^2\|\|x^{-2}\|+1}=$$
$$=\Delta(x)\left(1+\frac{2\|x\|\|x^{-1}\|}{\|x\|^2\|x^{-1}\|^2+1}\right)\leq 2\Delta(x).$$

$iv)$ Just use the inequality
$$\Delta(x\mathbf{1})\leq\max\{\Delta(x^2),\Delta(\mathbf{1})\}=\Delta(x^2).$$
 \end{proof}

\section{The centrality measure}

Let $\mathcal{N}$ be a $W^*$-algebra with the center $\mathfrak{C}(\mathcal{N})$ and with some representation $\pi: \mathcal{N} \mapsto B(H)$.

By $\Delta_\mathcal{N}: x\in \mathcal{N}^+\mapsto [0,1]$
we denote the following functional
$$\Delta_\mathcal{N}(x)=\inf\left\{\left\|z-\frac{x}{A}\right\|\ |\ A\in\mathbb{R}^+,z\in \mathfrak{C}(\mathcal{N}), \|z\|=1\right\}.$$

The $\Delta$ described in the previous section coincides with $\Delta_{B(H)}$.
Note,that $$\Delta_\mathcal{N}(x)\leq \Delta_{B(H)}(\pi(x)).$$

\begin{remark} If algebra $\mathcal{N}$ is simple, i.e. $\mathfrak{C}(\mathcal{N})=\mathbb{C}\cdot \mathbf{1}$, then, evidently, $$\Delta_\mathcal{N}\subset\Delta_{B(H)},$$
i.e. for any $x\in \mathcal{N}$ the equality  $\Delta_\mathcal{N}(x)=\Delta_{B(H)}(\pi(x))$ holds, thus it holds for any $W^*$-factor.
\end{remark}

However, it is easy to see, that the definitions are distinctive.
\begin{example}
 Let $\mathcal{N}=\ell_\infty$ be a commutative $W^*$-algebra. Then $\Delta_\mathcal{N}=0$ even for projections in $\ell_\infty$.
\end{example}

The general properties for $\Delta_\mathcal{N}$ are the same.

\begin{proposition}\label{prop3}
Let $\mathcal{N}$ be a $W^*$-algebra
and $x\in \mathcal{N}^+$, $\lambda\in\mathbb{R}^+$, then
\begin{enumerate}[1)]
 \item $\Delta_\mathcal{N}(\lambda x)=\Delta_\mathcal{N}(x);$
 \item $\Delta_\mathcal{N}(x)\leq 1;$
 \item if $x$ is invertible (non-singular), then $\Delta_\mathcal{N}(x)<1.$
\end{enumerate}
\end{proposition}
\begin{proof}
Statement 1) evidently follows from the definition, statement; 2) follows from the inequality $\Delta_\mathcal{N}(x)\leq \Delta_{B(H)}(\pi(x))$ and Proposition \ref{prop1}; 3) also follows from the same inequality and Theorem \ref{Theorem_formula}.
\end{proof}

\begin{theorem}\label{theorem5}
Let $\{\mathcal{N}_i\}_{i\in I}$ be a family $W^*$-algebras, and $\mathcal{M}$ be its direct sum  then for any $\oplus_{i\in I}\mathcal{N}_i$ (i.e. for any $a\in \mathcal{M}$ such that $a=\oplus_{i\in I} a_i$ with $a_i\in \mathcal{N}_i$ we have that $\sup_{i\in I} \{\|a_i\|\}<+\infty)$  the equality
$$\Delta_{\mathcal{M}}(\oplus_{i\in I} x_i)\leq \sup_{i\in I}\{\Delta_{\mathcal{N}_i}(x_i)\}$$
holds for any $a\in \mathcal{M}^+$ ($a_i\in\mathcal{N}_i^+)$
\end{theorem}
\begin{proof}
It is known that $\mathfrak{C}(\oplus_{i\in I} \mathcal{N}_i)=\oplus_{i\in I} \mathfrak{C}(\mathcal{N}_i)$. Also for any particular $x\in \oplus_{i\in I} \mathcal{N}_i$ with the proper decomposition $x=\oplus_{i\in I} x_i$ ($x_i\in \mathcal{N}_i$) we have $\|x\|=\|\oplus_{i\in I} x_i\|=\sup_{i\in I}\{\|x_i\|\}$, thus

$$\inf\left\{\left\|z-\frac{x}{A}\right\|\ |\ z\in\mathfrak{C}(\oplus_{i\in I}\mathcal{N}_i), A\in \mathbb{R}^+, \|z\|=1\right\}=$$
$$=\inf\left\{\left\|\oplus_{i\in I} z_i-\frac{\oplus_{i\in I} x_i}{A}\right\|\ |\ z_i\in\mathfrak{C}(\mathcal{N}_i), A\in \mathbb{R}^+, \sup_{i\in I}\|z_i\|=1\right\}=$$
$$=\inf\left\{\left\|\oplus_{i\in I}\left( z_i-\frac{x_i}{A}\right)\right\|\ |\ z_i\in\mathfrak{C}(\mathcal{N}_i), A\in \mathbb{R}^+, \sup_{i\in I}\|z_i\|=1\right\}=$$
$$=\inf\left\{\sup_{i\in I}\left\{\|z_i\|\left\| \frac{z_i}{\|z_i\|}-\frac{x_i}{A\|z_i\|}\right\| \right\} |z_i\in\mathfrak{C}(\mathcal{N}_i), A\in \mathbb{R}^+, \sup_{i\in I}\|z_i\|=1 \right\}\leq$$
$$\leq\sup_{i\in I}\left\{\inf\left\{\left\| z_i-\frac{x_i}{A_i}\right\|, {z_i}\in\mathfrak{C}(\mathcal{N_i}), A_i\in \mathbb{R}^+, \|z_i\|=1\ \right\}\right\}=\sup_{i\in I}\Delta_{\mathcal{N}_i}(x_i)$$
\end{proof}

Another definition that we propose and see it as more productive is the following:
any countably decomposable von Neumann algebra $\mathcal{N}$ is decomposable into direct sum of factors $\mathcal{N} = \oplus_{i\in I} F_i$.

For $x\in \mathcal{N}^+$ define $$\Delta_\mathcal{N}'(x):={\sup}_{i\in I}(\Delta(x_i)),$$
with $x_i\in F_i^+$.

\begin{proposition}\label{prop3}
Let $\mathcal{N}$ be a $W^*$-algebra
and $x\in \mathcal{N}^+$, $\lambda\in\mathbb{R}^+$, then
\begin{enumerate}[1)]
 \item $\Delta'_\mathcal{N}(\lambda x)=\Delta'_\mathcal{N}(x);$
 \item $\Delta'_\mathcal{N}(x)\leq 1;$
 \item if $x$ is invertible (non-singular), then $\Delta'_\mathcal{N}(x)<1.$
\end{enumerate}
\end{proposition}

\begin{remark}
Note, that any projection $p_j$ constructed as $p_j=\oplus_{i\in I} \delta_i^j \mathbf{1}$, where $\delta_i^j$ is Kronecker delta, actually is a central element, thus Proposition \ref{prop2} cannot be extended to the case of $\Delta_\mathcal{N}$ or $\Delta'_\mathcal{N}$.
\end{remark}
However we are able to generalize Theorem \ref{theorem5} for $\Delta'_\mathcal{N}$ with $\mathcal{N}$ being countably decomposable $W^*$-algebra

\begin{theorem}
Let $\mathcal{N}$ be countably decomposable von Neumann algebra and $x,y\in\mathcal{N}^+$, then
\begin{enumerate}[1)]
 \item $\Delta'_\mathcal{N}(xy)=\Delta'_\mathcal{N}(yx)$;
 \item $\Delta'_\mathcal{N}(xy)\leq\max(\Delta'_\mathcal{N}(x^2),\Delta'_\mathcal{N}(y^2));$
 \item $\Delta'_\mathcal{N}(x^2)\leq 2\Delta'_\mathcal{N}(x);$
 \item $\Delta'_\mathcal{N}(x)\leq\Delta'_\mathcal{N}(x^2).$
\end{enumerate}
\end{theorem}
\begin{proof}
Simply consider two representations $x=\oplus_{i=1}^n x_i$, $y=\oplus_{i=1}^n y_i$, where $x_i,y_i\in F_i^+$, then

\begin{enumerate}[1)]
    \item $$\Delta'_\mathcal{N}(xy) =\sup_{i\in I}\Delta(x_iy_i)=\sup_{i\in I}\Delta(y_ix_i)=\Delta'_\mathcal{N}(yx);$$c
    \item $$\Delta'_\mathcal{N}(xy)=\sup_{i\in I}\Delta(x_iy_i)\leq \sup_{i\in I}\{\max\{\Delta(x_i^2),\Delta(y_i^2)\}\}=$$
    $$= \max\{\sup_{i\in I}\Delta(x_i^2),\sup_{i\in I}\Delta(y_i^2)\}\}=\max\{\Delta'_\mathcal{N}(x),\Delta_\mathcal{N}(y)\};$$
    \item $$\Delta'_\mathcal{N}(x^2)=\sup_{i\in I}\Delta(x_i^2)\leq \sup_{i\in I} 2\Delta(x_i) =2 \sup_{i\in I} \Delta(x_i)=2\Delta'_\mathcal{N}(x);$$
    \item $$\Delta_\mathcal{N}'(x)=\sup_{i\in I}\Delta(x_i)\leq \sup_{i\in I}\Delta(x_i^2)=\Delta'_\mathcal{N}(x^2).$$
\end{enumerate}
\end{proof}
Also, we can generalize Theorem \ref{Theorem4} for $\Delta'_\mathcal{N}$.
\begin{theorem}\label{theorem7}
Let $\mathcal{M}$ be countably decomposable von Neumann algebra and $x,y\in\mathcal{M}^+$, then $$\Delta'_\mathcal{M}(x+y)\leq \max\{\Delta'_\mathcal{M}(x), \Delta'_\mathcal{M}(y)\}.$$
\end{theorem}
\begin{proof}

$$\Delta'_\mathcal{N}(x+y)=\sup_{i\in I}\left( \Delta(x_i+y_i)\right)\leq \sup_{i\in I}\left( \max\{\Delta(x_i), \Delta(y_i)\}\right)=$$
$$=\max\left\{\sup_{i\in I}\left(\Delta(x_i)\right),\sup_{i\in I}\left(\Delta(y_i)\right)\right\}=\max\left\{\Delta'_\mathcal{N}(x),\Delta'_\mathcal{N}(y)\right\}$$
\end{proof}

Based on the latter Theorem and Proposition \ref{prop3} for $\mathcal{N}$ we can define a cone $$K_c=\{x\in \mathcal{N}^+ | \Delta'_\mathcal{N}(x)\leq c \},$$ which would be a closed subcone with the property, that if $0\leq c_1\leq c_2\leq 1$, then $$\mathbb{R}^+\mathbf{1}\subset K_0\subset K_{c_1}\subset K_{c_2}\subset K_1=\mathcal{M}^+.$$

\section{Applications}

\subsection{Gardner's inequality inspired simulation}
In 1979 L.T. Gardner showed the inequality $|\varphi(X)|\leq \varphi(|X|)$ characterizes traces in $C^*$-algebras among all functionals, i.e.

\begin{theorem}[\cite{Gardner1979}, Theorem 1]
The finite traces on a $C^*$-algebra $\mathcal{A}$ are precisely those (positive) linear functional s $\varphi$ on $\mathcal{A}$ which satisfy $|\varphi(x)|\leq \varphi(|x|)$ for all $x\in \mathcal{A}$.
\end{theorem}

If $\varphi$ is a traceable functional on the $C^*$-algebra $\mathcal{A}$, then the Gardner exponent shows the result for all elements of $X\in\mathcal{A}$ and, conversely, if for all $ X\in \mathcal{A}$ is a Gardner quality indicator and $\varphi$ is a positive functional, this functional is traceable.

Let $\mathcal{M}$ be a von Neumann algebra, the normal strongly semifinite weight $\varphi$ ensures that for any $\varphi$-finite projects $P\in\mathcal{M}$, the Gardner equivalent ( $|\varphi(X)|\leq \varphi(|X|)$) result for all $X=PX_0P$, where $X_0\in \mathcal{M}$, then the weight is a trace.

In case if $\mathcal{A}=\mathbb{M}_n(\mathbb{C})$, we have that for $\varphi:=\mathrm{Tr}(A\cdot)$ the inequality must be violated for some $X\in \mathbb{M}_n(\mathbb{C})$, i.r. exists $X\in\mathbb{M}(\mathbb{R})$ such that $|\mathrm{Tr}(AX)|-\mathrm{Tr}(A|X|)>0$.

We made the Monte-Carlo simulations and got the following scatter plots for $n=\overline{2,3}$.

\begin{figure}[h!]
    \centering
    \includegraphics[width = \textwidth]{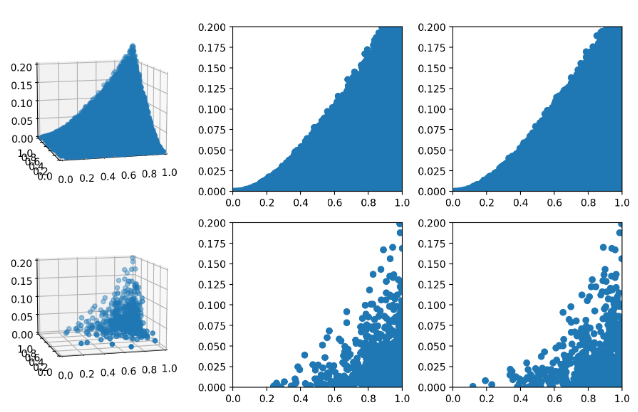}
    \caption{$x=\Delta(X)$, $y=\Delta(Y)$, with $X\in\mathbb{M}^+_n(\mathbb{R}), Y\in \mathbb{M}_n(\mathbb{R}), \|X\|=\|Y\|=1$ and $z=|\mathrm{Tr}(XY)|-\mathrm{Tr}(X|Y|)$. The left column is a 3D scatter plot, the middle column is a plot of $z$ vs. $x$ and the right column is $z$ vs. $y$, the upper row is for $n=2$ and the lower row is for $n=3$.}
    \label{fig:enter-label}
\end{figure}

Another simulation involves $\mathrm{Tr}(|XY|) - \mathrm{Tr}(|X||Y|)$.

\begin{figure}[h!]
    \centering
    \includegraphics[width = 0.8\textwidth]{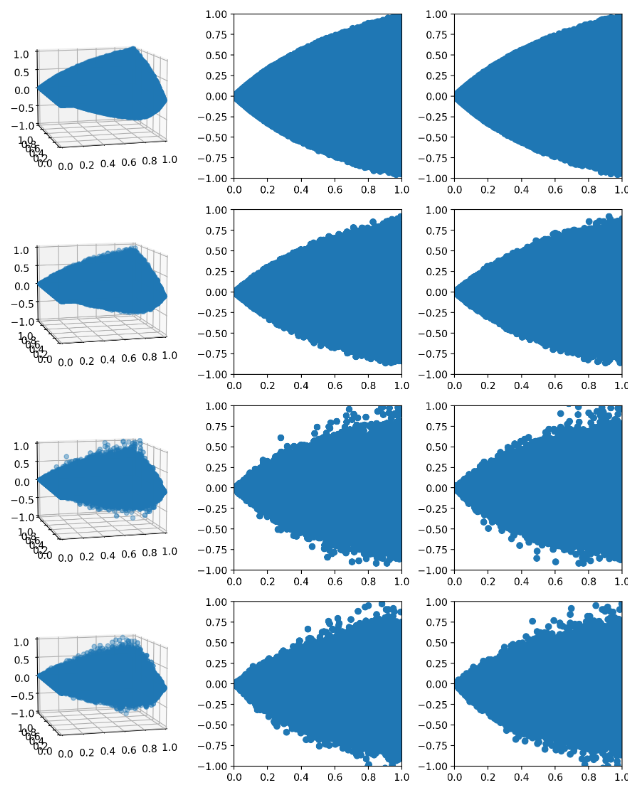}
    \caption{The scatter plots above are visualising results of simulations with $x=\Delta(X)$, $y=\Delta(Y)$, with $X\in\mathbb{M}_n(\mathbb{R}), Y\in \mathbb{M}_n(\mathbb{R}), \|X\|=\|Y\|=1$ and $z=|\mathrm{Tr}(XY)|-\mathrm{Tr}(|X||Y|)$. The left column is a 3D scatter plot, the middle column is a plot of $z$ vs. $x$ and the right column is $z$ vs. $y$. The rows correspond for $2$, $3$, $4$ and $5$-dimentional simulations respectively.}
    \label{fig:enter-label}
\end{figure}

\subsection{$L_1$ equality violation}

From \cite{Nov2017} we now that if $A\in \mathbb{M}^+_n(\mathbb{R})$ and $\mathrm{Tr}(|AXA|) = \mathrm{Tr}(A|X|A)$ for all $X\in\mathbb{M}_n^{sa}(\mathbb{R})$ then $A$ is central.

\begin{figure}[h!]
    \centering
    \includegraphics[width = 0.8\textwidth]{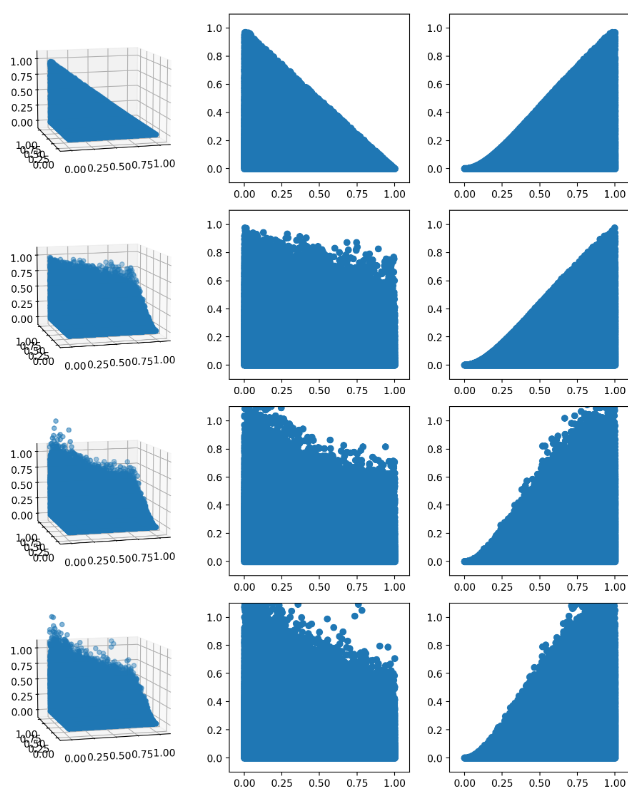}
    \caption{The scatter plots above are visualising results of simulations with $x=\Delta(|X|)$, $y=\Delta(Y)$, with $X\in\mathbb{M}_n^{sa}(\mathbb{R}), Y\in \mathbb{M}_n^+(\mathbb{R}), \|X\|=\|Y\|=1$ and $z=\mathrm{Tr}(Y|X|Y)|-\mathrm{Tr}(|YXY|)$. The left column is a 3D scatter plot, the middle column is a plot of $z$ vs. $x$ and the right column is $z$ vs. $y$. The rows correspond for $2$, $3$, $4$ and $5$-dimentional simulations respectively.}
    \label{fig:enter-label}
\end{figure}

\section{Data Avaialabilty Statement}

All of the code for simulations is available by the following link

$\href{https://www.kaggle.com/code/andreinovikov90/computations-for-michelson-contrast}{https://www.kaggle.com/code/andreinovikov90/computations-for-michelson-contrast}$

%\href{https://drive.google.com/drive/folders/1C0GCDjVlpxxy1nplEzsyh0-u_Np_a66i?usp=sharing}
%{https://drive.google.com/drive/folders/1C0GCDjVlpxxy1nplEzsyh0-u_Np_a66i?usp=sharing}

along with the outputs of the simulations. The code alone is also available on GitHub:

$\href{https://github.com/AHHobukob/michelson_contrasts_simulation_2024}
{https://github.com/AHHobukob/michelson\_contrasts\_simulation\_2024}.$

\section{Discussion}

It appears that the proven properties of the Michelson contrast were previously unknown, or at least not proven. A brief overview of the applications gives us a clear idea that these results can be used both theoretically and practically.

From the simulations, it seems that the inequalities that characterize traces and central elements are violated for an amount that can have upper bounds dependent on Michelson contrasts. Also, we think that these properties can be generalized for the general case of $C^*$-algebras and for the weights on $W^*$-algebras.

It is notable that from the Dixmier characterization of positive operators we obtained a Michelson contrast for operators. It becomes even more exciting when we refer to \cite{BRV20120} where authors showed the equivalence between Jensen–Shannon
divergence and Michelson contrast for the commutative distributions. If the general property has a physical fundamentals, then the equivalence musti also occure for the quentum Jensen-Shannon divergence.\cite{Holevo1975}

Also, it is interesting to find out how does the Michelson contrast of the components affect the inequalities characterizing the traces, for instance the monotonicity inequalities \cite{BikTik, B&T, B&T2, Virosztek}, subadditivity \cite{Tik2005, TikShe}, Young's inequality \cite{Cho2009} and other characterizations of traces \cite{Bik2013, Bik2015, Bik2020, PZ88, Sano2006}. You can take a look at the survey \cite{Bik2010add1}.

Also, it seems that we can try to prove limit theorems for $\Delta'\left(\frac{1}{n}\sum_{k=1}^n X_k\right)$ considering the sequences $X_1,X_2,\dots, X_n\dots$ of random sample elements from the population from the quantum probability space $(\mathcal{N},\mathcal{N}^{pr}, \varphi)$, where $\mathcal{N}$ is a countably decomposable von Neumann algebra, yet it seems to work in simulations for random matrices when $\mathcal{N}=\mathbb{M}_n(\mathbb{R})$.

Another application seems to be found in optics and vision theory.  We can think of the raster graphical image as a matrix in $\mathbb{M}_{n\times m}(\mathbb{Z}_d)$ for monochromatic image of $d$ pixel depth or $\oplus_{i=1}^{c}\mathbb{M}_{n\times m}(\mathbb{Z}_d)$, where $c$ is a number of channels and $d$ is a channel depth of an image (respectively, $3$ and $256$ for the RGB color space), which can be embedded into $\mathbb{M}_{n\times m}(\mathbb{R})$ and $\oplus_{i=1}^{c}\mathbb{M}_{n\times m}(\mathbb{R})$, respectively.

From here we may move in different directions, first is to consider a simple
assumption that the image is described as a function $f(x,y)\geq 0$ and we consider a multiplication operator $M_f: g\in L_2(\mathbb{Z}_n\times\mathbb{Z}_m)\mapsto fg\in L_2(\mathbb{Z}_n\times\mathbb{Z}_m)$ we have that Michelson distance is just equal to $\Delta(M_f)$.

Another direction is to extend the Michelson contrast for the rectangular matrices through the singular values framework.

We know, that since $\Delta$ and $\Delta'$ are dependant on the spectrum, then it is not changing under unitary transformations, however, we have nonlinear transformations between color schemes such as RGB and CMYK, which can give us a raise of contrast, that could help in computer vision to highlight some specific objects. Also, it can be interesting how the activation functions affect the Michelson contrast values, that could bring the light upto fundamentals of neural networks.

\section*{Acknowledgments}

The contribution of the first author (Dr. Sami Abdullah Abed) in this work was supported by the Ministry of Higher Education and Scientific Research of the Republic of Iraq.

The contribution of the second author (Irina Nikolaeva) in this work was supported by the development program of the Volga Region Mathematical Center under agreement No. 075-02-2020-1478 with the Ministry of Science and Higher Education of the Russian Federation.

The contribution of the third author (Dr. Andrej Novikov) is supported by the Mathematical Center in Akademgorodok under agreement No.075-15-2022-281 (05.04.2022) with the Ministry of Science and Higher Education of the Russian Federation.

\end{document}